\documentclass{amsart}
\usepackage[utf8]{inputenc}

\usepackage{tikz-cd}
\usepackage{pgf,tikz}
\usepackage{mathrsfs}
\usetikzlibrary{arrows}

\usepackage{graphicx}
\usepackage{subfig}

\usepackage{hyperref}

\usepackage{latexsym}
\usepackage{xcolor}
\usepackage{tikz}
\hypersetup{
	colorlinks,
	linkcolor={blue},
	citecolor={blue},
	urlcolor={blue}
}
\usepackage[mathcal]{euscript}
\usepackage{mathrsfs}
\usepackage{amsmath}
\usepackage{youngtab}
\usepackage{graphicx}

\usepackage{amsmath, amssymb}
\usepackage{bbm}
\usepackage{xcolor}

\usepackage[
backend=biber,
style=alphabetic,
sorting=ynt
]{biblatex}
\usepackage[inner=3cm, outer=3cm, top=3.5cm]{geometry}

\newcommand{\R}{\mathbb{R}}
\newcommand{\RP}{\mathbb{R}\textrm{P}}

\newcommand{\CP}{\mathbb{C}\textrm{P}}
\newcommand{\C}{\mathbb{C}}

\newcommand{\G}{\mathbb{G}}

\newcommand{\dd}{\mathrm{d}}

\newtheorem{thm}{Theorem}
\newtheorem{lemma}[thm]{Lemma}
\newtheorem{cor}[thm]{Corollary}
\newtheorem{prop}[thm]{Proposition}

\newtheorem{df}{Definition}

\theoremstyle{remark} 
\newtheorem{rk}[]{Remark}

\newcommand{\beq}{\begin{equation}}
\newcommand{\eeq}{\end{equation}}
\newcommand{\ubeq}{\begin{equation*}}
\newcommand{\ueeq}{\end{equation*}}

\usepackage{mathtools} 
\mathtoolsset{showonlyrefs,showmanualtags} 

\DeclareMathOperator{\Diff}{D}
\DeclareMathOperator{\diff}{d}

\newcommand{\Esp}{\mathbb{E}}

\newcommand{\Prj}{\mathrm{P}}
\newcommand{\Real}{\mathbb{R}}

\newcommand{\grpsym}{\mathfrak{S}}
\newcommand{\Grass}{\mathbb{G}}

\title{Probabilistic Schubert Calculus: asymptotics}
\author{Antonio Lerario and Leo Mathis}

\begin{document}

\begin{abstract}
In the recent paper \cite{PSC} Bürgisser and Lerario introduced a geometric framework for a probabilistic study of real Schubert Problems. They denoted by $\delta_{k,n}$ the average number of projective $k$-planes in $\Real\Prj^n$ that intersect $(k+1)(n-k)$ many random, independent and uniformly distributed linear projective subspaces of dimension $n-k-1$. They called $\delta_{k,n}$ the \emph{expected degree} of the real Grassmannian $\Grass(k,n)$ and, in the case $k=1$, they proved that:
$$ \delta_{1,n}= \frac{8}{3\pi^{5/2}} \cdot \left(\frac{\pi^2}{4}\right)^n \cdot n^{-1/2} \left( 1+\mathcal{O}\left(n^{-1}\right)\right)
.$$
Here we generalize this result and prove that for every fixed integer $k>0$ and as $n\to \infty$, we have
\begin{equation*}
\delta_{k,n}=a_k \cdot \left(b_k\right)^n\cdot n^{-\frac{k(k+1)}{4}}\left(1+\mathcal{O}(n^{-1})\right)
\end{equation*}
where $a_k$ and $b_k$ are some (explicit) constants, and $a_k$ involves an interesting integral over the space of polynomials that have all real roots.  For instance:
$$\delta_{2,n}= \frac{9\sqrt{3}}{2048\sqrt{2\pi}} \cdot 8^n \cdot n^{-3/2} \left( 1+\mathcal{O}\left(n^{-1}\right)\right).$$
Moreover we prove that these numbers belong to the ring of \emph{periods} intoduced by Kontsevich and Zagier and we give an explicit formula for $\delta_{1,n}$ involving a one dimensional integral of certain combination of Elliptic functions.
\end{abstract}

\maketitle

\section{Introduction}
\subsection{Random Real Enumerative Geometry}

In this paper we continue the study of real enumerative problems initiated in~\cite{PSC}. Our goal is to answer questions such as

\begin{center} \emph{In average, how many lines intersect four \emph{random} lines in $\RP^3$?} \end{center}

To be more precise, let $\Grass(1,3)$ be the Grassmannian of lines in $\Real\Prj^3$. This is a homogeneous space with a transitive action of the orthogonal group $O(4)$ on it and there is a unique invariant probability measure defined on $\G(1,3)$ invariant under this action. We fix $L\in\Grass(1,3)$ and define the \emph{Schubert variety}:
\begin{equation} 
\Omega(L):=\left\{\ell\in\Grass(1,3) \ | \ \ell\cap L \neq \emptyset \right\}.
\end{equation}

Then the solution to the problem is the number:
\begin{equation}
\delta_{1,3}:=\Esp{ \# \left(g_1\cdot\Omega\left(L\right) \cap \ldots \cap g_4\cdot\Omega\left(L\right)\right)}
\end{equation}
where $g_1,\ldots,g_4$ are independent, taken uniformly at random from $O(4)$ (with the normalized Haar measure).\\

One can generalize to higher dimensions. Let $\Grass(k,n)$ be the Grassmannian of linear projective subspaces of dimension $k$ in $\Real\Prj^n$. It is a homogeneous space with $O(n+1)$ acting transitively on it and with a unique $O(n+1)-$invariant probability measure. We fix $L\in\Grass(n-k-1,n)$ and introduce the corresponding \emph{Schubert variety}:
\begin{equation} \label{defSchub}
\Omega(L):=\left\{\ell\in\Grass(k,n) \ | \ \ell\cap L \neq \emptyset \right\}.
\end{equation}
We define
\begin{equation}\label{eq:ed}
\delta_{k,n}:=\Esp{ \# \left(g_1\cdot\Omega\left(L\right) \cap \ldots \cap g_{(k+1)(n-k)}\cdot\Omega\left(L\right)\right)}
\end{equation}
where $g_1,\ldots,g_{(k+1)(n-k)}$ are independent, taken uniformly at random from $O(n+1)$ (with the normalized Haar measure). This number equals the average number of $k$-dimensional subspaces of $\RP^n$ meeting $(k+1)(n-k)$ random subspaces of dimension $(n-k-1)$.

\subsection{Previously on \emph{Probabilistic Schubert calculus}}\label{subsec:Previously}

In the recent work \cite{PSC}, the first named author of the present paper together with Peter B\"urgisser established a formula\footnote{Notice that, in the language of \cite{PSC}, $\delta_{k,n}=\mathrm{edeg}(G(k+1, n+1)).$} for the number $\delta_{k,n}$ in \eqref{eq:ed}, see \cite[Corollary 5.2]{PSC}:
\begin{equation}
\delta_{k,n}=\frac{d_{k,n}!}{2^{d_{k,n}}} \cdot |\Grass(k,n)| \cdot |C(k,n)|
\end{equation}
where $|\Grass(k,n)|$ is the volume of the Grassmanniann and $|C(k,n)|$ the volume of a certain convex body in $\Real^{(k+1)\times (n-k)}$, which the authors called the \emph{Segre zonoid}. This convex body is defined as follows: take random points $p_1,\ldots,p_m$ independently and uniformly on $S^k\times S^{n-k-1} \subset \Real^{(k+1)(n-k)}$ and consider the Minkowski sum $K_m:=\frac{1}{m}([0,p_1]+\cdots+[0,p_m])$. Then $C(k,n)$ is the limit, with respect to the Hausdorff metric, of $K_m$ as $m\to\infty$ (being a limit of zonotopes, $C(k,n)$ is a zonoid).

If we see elements of $ \Real^{(k+1)\times(n-k)}$ as matrices it turns out that this convex body, in some sense only depends on their \emph{singular values}. Using this and assuming $k+1<n-k$ one can construct a convex body in the space $\R^{k+1}$ of singular values such that if we call $r$ its radial function, we have \cite[Theorem 5.13]{PSC}:
\begin{equation}\label{intro:edegint}
\delta_{k,n}=\beta_{k,n}\int_{S^{k}_+}{\left(p_k \cdot (r)^{k+1}\right)^{(n-k)}q_k \ dS^{k}}
\end{equation}
where $p_k$ and $q_k$ are simple combinatorial functions of the coordinates on $\Real^{k+1}$, $\beta_{k,n}$ is a known coefficient (whose explicit expression is given in \eqref{eq2}) and the domain of integration is 
\beq S^{k}_+=\left\{x\in\Real^{k+1}|\ \|x\|=1,\, x_1\geq\ldots\geq x_{k+1} \geq 0\right\}.\eeq
Equation~\eqref{intro:edegint} will be our starting point for computing both the asymptotic of $\delta_{k,n}$ and the ``exact" formula for $\delta_{1,3}$.

\subsection{Main Results}
Our first main result is the asymptotic of $\delta_{k,n}$ for any fixed $k$, as $n$ goes to infinity, generalizing \cite[Theorem 6.8]{PSC}, which deals with the case $k=1$.
\begin{thm}\label{thm:asintro}
For every integer $k>0$ and as $n$ goes to infinity, we have
\begin{equation*}
\delta_{k,n}=a_k \cdot \left(b_k\right)^n\cdot n^{-\frac{k(k+1)}{4}}\left(1+\mathcal{O}(n^{-1})\right)
\end{equation*}
where
\begin{align*}
a_k &=\Lambda_k \  \frac{2^{\frac{(k+1)(k-2)}{4}}}{\pi^{\frac{k(k+2)}{2}}}\ \frac{\Gamma\left(\frac{k(k+3)}{4} \right)}{\Gamma\left(\frac{k(k+1)+2}{4} \right)}\left( \frac{k+1}{k+2}\right)^{\frac{k(k+3)}{4}}\left(\frac{\Gamma\left(\frac{k+1}{2} \right)}{\Gamma\left(\frac{k+2}{2} \right)}\right)^{k(k+1)}         \label{defak intro} \\
b_k&=\left( \frac{\Gamma\left(\frac{k+2}{2} \right)}{\Gamma\left(\frac{k+1}{2} \right)}\sqrt{\pi}\right)^{(k+1)}.
\end{align*}
(The number $\Lambda_k$ that appears in the expression of $a_k$ can be expressed as an integral over the polynomials that have all roots in $\Real$, see Definition \ref{def Lambdak}.)
\end{thm}

For instance $\Lambda_1$ and $\Lambda_2$ can be easily computed and the previous formula gives:
\begin{align}
\delta_{1,n}&= \frac{8}{3\pi^{5/2}} \cdot \left(\frac{\pi^2}{4}\right)^n \cdot n^{-1/2} \left( 1+\mathcal{O}\left(n^{-1}\right)\right)\\
\delta_{2,n}&= \frac{9\sqrt{3}}{2048\sqrt{2\pi}} \cdot 8^n \cdot n^{-3/2} \left( 1+\mathcal{O}\left(n^{-1}\right)\right).
\end{align}

Similarly one can consider the same problem over the \emph{complex} Grassmannian $\G_\C(k,n)$ of $k$-dimensional complex subspaces of $\CP^n$. The Schubert cycles $\Omega_\C(L_\C)$ are defined just as in~\eqref{defSchub}. The compact Lie group with transitive action is now the unitary group $U(n+1)$. We define
\beq \label{dkncIntro}
\delta_{k,n}^\C:=\Esp{ \# \left(g_1\cdot\Omega_\C(L_\C) \cap \ldots \cap g_{(k+1)(n-k)}\cdot\Omega_\C(L_\C)\right)}
\end{equation}
where $g_1,\ldots,g_{(k+1)(n-k)}$ are independent, taken uniformly at random from $U(n+1)$ (with the normalized Haar measure).
\begin{rk}
The expected value in~\eqref{dkncIntro} is an integer. Indeed the variable is almost surely constant and computes the \emph{degree} of the Grassmannian in the Pl\"ucker embedding (see~\cite[Corollary 4.15]{PSC}).
\end{rk}
For this number we derive the following asymptotic (to be compared with Theorem~\ref{thm:asintro}).
\begin{prop}[The asymptotic for the complex case] 

\beq
\delta_{k,n}^\C = a_k^\C \cdot \left(b_k^\C\right)^n\cdot n^{-\frac{k(k+2)}{2}}\left(1+\mathcal{O}(n^{-1})\right)
\label{dknCasymp}
\eeq
where
\begin{align}
a_k^\C 	&=\frac{\Gamma(1)\Gamma(2)\cdots \Gamma(k+1)}{(2\pi)^{k/2}(k+1)^{k(k+1)-1/2}}         \label{defakC} \\
b_k^\C	&=\left(k+1\right)^{(k+1)}.
\end{align}

\end{prop}
\begin{rk}To derive the asymptotic formula of Theorem \ref{thm:asintro}, we first notice that $\mu=\frac{1}{\sqrt{k+1}}(1,\ldots,1)$ is the only critical point of $r$ (and $p_k$) in the domain of integration of~\eqref{intro:edegint}, it is a maximum and is non degenerate. Thus if we can compute its Hessian $H_k$ at this point we could compute the asymptotic of~\eqref{intro:edegint} using Laplace's method.

The difficulty lies in the fact that $H_k$ is a symmetric bilinear form on $T_\mu S^{k} \cong\Real^{k}$, thus we would need to compute $\sim k^2$ entries for large $k$. However here we are saved by the symmetries of the convex body whose radial function is $r$. Indeed it is invariant by permutation of coordinates in $\Real^{k+1}$. This implies that $H_k$ commutes with this action of the symmetric group $\grpsym_{k+1}$. Moreover $T_\mu S^{k}=\mu^\perp$ is an irreducible subspace for this action. Thus by Schur's Lemma $H_k=\lambda_k \cdot \mathbf{1}$ for some $\lambda_k\in\Real$: in this way, for each $k>0$, we only need to compute one number! Still this computation is non trivial (see Proposition~\ref{prop:Hessr}).
\end{rk}
It is not difficult to prove (Corollary \ref{cor:periods1} below) that $\delta_{k,n}$ belongs to the ring of \emph{periods} introduced by Kontsevich and Zagier. Other than this, the nature of these numbers remains mysterious. In fact we do not even have an ``exact'' formula for the simplest non trivial case $\delta_{1,3}$. Nevertheless we can present it as a one-dimensional integral (see Proposition \ref{prop:periods2}).

\begin{thm} \label{intro:thmlineInt}

\beq\delta_{1,n}=-2 \pi^{2n-2}c(n)\int_{0}^1 L(u)^{n-1}\mathrm{sinh}(w(u))w'(u)du
\eeq
where
\beq
c(n)=\frac{\Gamma\left(2n-2\right)}{\Gamma\left(n\right)\Gamma\left(n-2\right)}
\eeq
 $L=F\cdot G$ and $w=\log\left(F/G\right)$ with
\begin{align}
F(u)&:=\int_{0}^{\pi/2} \frac{u \ \sin^2(\theta)}{\sqrt{\cos^2\theta+u^2 \sin^2\theta}}\mathrm{d}\theta	\\
G
(u)&:=\int_{0}^{\pi/2} \frac{u \ \sin^2(\theta)}{\sqrt{\sin^2\theta+u^2 \cos^2\theta}}\mathrm{d}\theta
\end{align}

\end{thm}

\begin{rk}
One may want to evaluate numerically $\delta_{k,n}$\footnote{This is in fact the topic of a discussion on Mathoverflow: https://mathoverflow.net/questions/260607/expected-number-of-lines-meeting-four-given-lines-or-what-is-1-72}. For this purpose Equation~\eqref{intro:edegint} isn't quite suitable because we don't know the radial function $r$ explicitly. 

On the opposite in Theorem~\ref{intro:thmlineInt} everything is explicit. The functions $F$ and $G$ are elliptic integrals and satisfy a rather simple linear differential equation with rational coefficient. One could then use techniques such as the D-modules machinery (see for example~\cite{Dmod}) to obtain numerical evaluation.
\end{rk}

\subsection{Structure of the paper}
In~\cite{PSC} is initiated the study of the numbers $\delta_{k,n}$. We will first recall what is achieved there as well as some preliminary background in Section~\ref{sec:Prel}. In Section~\ref{sec:Asymp} we compute the asymptotic of $\delta_{k,n}$ as $n$ goes to infinity; this is to be compared with the asymptotic in the \emph{complex}, case which is computed in Section~\ref{sec:Cpx}. In section~\ref{sec:periods} we prove that $\delta_{k,n}$ is a period in the sense of Kontsevitch Zagier. Finally in Section~\ref{sec:lineInt} we provide a formula for $\delta_{1,n}$ for every $n\geq3$ as a one dimensional integral of elliptic functions.

\subsection{Acknowledgment}
We would like to thank Erik Lundberg for the very fruitful discussions that lead to Proposition~\ref{prop:periods2}. We also thank Don Zagier for his interesting remarks on $\delta_{1,3}$.

\section{Preliminaries}\label{sec:Prel}

\subsection{The Gamma Function}

\begin{df}
The \emph{Gamma Function} is defined for all $x>0$ by
\beq
\Gamma(x)=\int_0^{+\infty} t^{x-1}e^{-t} \dd t
\eeq
\end{df}
We will use the two following classical results. For a proof and more details see for example~\cite{AsymGamma}.

\begin{prop}\label{prop:GammaAsymp} For all real numbers $a$ and $b$
\begin{equation} \label{Frac Gam asymptote}
\frac{ \Gamma (x+a)}{ \Gamma (x+b)}=x^{a-b}\left(1+\mathcal{O}\left({x}^{-1}\right)\right).
\end{equation}
\end{prop}

\begin{prop}[Multiplication Theorem]\label{prop:MultThm} For all $x\geq0$ and for all integer $m$ we have
\beq\prod _{k=0}^{m-1}\Gamma \left(x+{\frac {k}{m}}\right)=(2\pi )^{\frac {m-1}{2}}\;m^{{\frac {1}{2}}-mx}\;\Gamma (mx). \eeq
\end{prop}

\subsection{The Grassmannian} \label{homgrass}

\begin{df}
The \emph{real Grassmannian} manifold is the homogeneous space 
\begin{equation*}
\Grass(k,n):=\frac{O(n+1)}{O(n-k)\times O(k+1)}
\end{equation*}
where $O(m)$ is the orthogonal group of $m\times m$ orthogonal matrices.
It is a smooth manifold of dimension 
\beq\dim(\Grass(k,n))=d_{k,n}:=(k+1)(n-k).\eeq
\end{df}

\begin{df}
The \emph{Plücker embedding} is the embedding
\begin{equation}
\begin{split}
\Grass(k,n) 	&\to \Prj(\Lambda^{k+1} \Real^{n+1}) \\
W		&\mapsto \left[a_1 \wedge \ldots \wedge a_{k+1} \right] \label{Plu}
\end{split}
\end{equation}
where $\left\{a_1,\ldots,a_{k+1}\right\}$ is any basis for $W$.
\end{df}
We provide $\Grass(k,n)$ with the Riemannian structure induced by \eqref{Plu}, recalling that the scalar product on $(k+1)$-vectors is given by:
\begin{equation}
\langle u_1\wedge \ldots \wedge u_{k+1},\  v_1\wedge \ldots \wedge v_{k+1}\rangle=\det \left( \langle u_i, v_j\rangle\right)_{1\leq i,j\leq k+1} .
\end{equation}

The volume of the Grassmannian with respect to the volume density associated to the the restriction of the Pl\"ucker metric is \cite[Equation (2.11) and (2.14)]{PSC}:
\beq\label{eq:volG}
\left|\G(k,n) \right|=\frac{\left|O(n+1)\right|}{\left|O(k+1)\right|\left|O(n-k)\right|}=\pi^{\frac{(k+1)(n-k)}{2}}\prod_{i=1}^{k+1}\frac{\Gamma\left(\frac{i}{2}\right)}{\Gamma\left(\frac{n-k+i}{2}\right)}
\eeq

\begin{rk} \label{dim codim}
From~\eqref{homgrass} we see that $\Grass(k,n)=\Grass(n-k-1,n)$. Thus we can and will assume for now on $k+1\leq n+1$.
\end{rk}

\subsection{Convex bodies}

We will need a few elementary results from convex geometry.

\begin{df}
A \emph{convex body}, is a non empty compact convex subset of $\Real^n$. We denote by $\mathscr{K}_n$ the set of convex bodies of $\Real^n$ containing the origin.
\end{df}

 \begin{df}
The \emph{support function} of $K\in\mathscr{K}_n$ is the function 
\begin{align*}
h_K : \Real^n &\to \Real \\
           x         &\mapsto \max\{\langle x,y \rangle \  |\  y\in K\}.
\end{align*}
\end{df}

\begin{df}
The \emph{support hyperplane} $H(K,u)$ of $K\in\mathscr{K}_n$ in the direction $u\in S^{n-1}$ is
\begin{equation*}
H(K,u)\coloneqq \left\{x\in\Real^n \ | \ \langle x,u \rangle=h_K(u)\right\}.
\end{equation*}
\end{df}

\begin{figure}
\centering
\includegraphics[width=5.73cm,height=4.28cm]{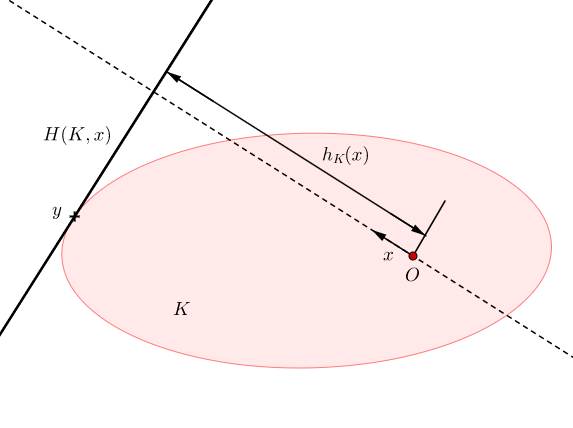}
\caption{The support function}
\label{figsupp}
\end{figure}

Intuitively, the support function, or more precisely its restriction to the sphere $S^{n-1}$, associates to each direction $u\in S^{n-1}$ the distance to the hyperplane $H(K,u)$, see Figure~\ref{figsupp}.
It charcaterizes the body in the sense that $h_{K_1}=h_{K_2}\iff K_1=K_2$. Moreover it satisfies some nice properties making it very useful, see \cite[Section 1.7.1]{SchConv} for proofs and more details:
\begin{align*}
K_1=K_2 &\iff h_{K_1}=h_{K_2} &
 K_1\subset K_2 &\iff h_{K_1}\leq h_{K_2} \\
h_{K_1+K_2}&=h_{K_1}+h_{K_2} &
 h_{\lambda K}&=\lambda h_{K}
 \end{align*}
 \begin{equation*}
x\in K \iff \langle x,y \rangle \leq h_K(y) \ \forall y \in \Real^n
\end{equation*}

The following result will be useful for us, see \cite[Corollary 1.7.3]{SchConv}.
\begin{prop} 
If $h_K$ is differentiable in $x_0 \in \Real^n$, then 
\begin{equation*}
\{\nabla h_K(x_0)\}= \{y\}=\partial K \cap H\left(K,\frac{x_0}{\|x_0\|}\right).
\end{equation*}
\label{gradh}
\end{prop}
 
 We will also need an other function representing convex bodies.
 
 \begin{df}
 The radial function of $K\in\mathscr{K}_n$ is
 \begin{align*}
r:S^{n-1} &\to            \Real_+ \\
   u     \     &\mapsto   \sup\left\{t\geq0 \ | \ t u \in K\right\}.
\end{align*}
 \end{df}
In this paper we will be interested in a special class of convex bodies: these are zonoids associated to a probability distribution in $\mathbb{R}^d$. This correspondence between zonoids and probability measures is studied in~\cite{Vit}. See for example~\cite[Theorem 3.1]{Vit}. We introduce the following definition.
\begin{df}[Vitale zonoid]Let $v\in \mathbb{R}^d$ be a random vector such that $\mathbb{E}\|v\|<\infty.$ We define the \emph{Vitale zonoid} associated to $v$ to be the convex body with support function $h(u)=\mathbb{E}h_{[0, v]}(u).$ 
\end{df}

There is a special case of this construction that will be relevant for us. Let $Z\subset \mathbb{R}^d$ be a compact semialgebraic set, and sample $v$ at random from the uniform distribution\footnote{This means the following. First restrict the Riemannian metric of $\mathbb{R}^d$ to the set of smooth points of $Z$, and consider the corresponding volume density. The total volume of the set of smooth points with respect to this density is finite, and we can normalize it to be equal to $1$. In this way $Z$ becomes a probability space (singular points have probability zero). We call the resulting probability distribution the uniform distribution on $Z$.} on $Z$. Then we will denote by $C_Z$ the Vitale zonoid associated to $Z$.
\subsection{Laplace's method}
The main step for the computation of the formula~\eqref{dknasymp} is to apply an asymptotic method for computing integrals, the so called Laplace's method -- in a multidimensional setting. For a proof and more details on this result, one can see \cite[Section II Theorem 1]{WongAsymp} . 

\begin{thm}[Laplace's method] \label{thm:Laplace}
We consider the integral depending on one parameter $\lambda>0$:
\begin{equation*}
I(\lambda)\coloneqq \int_{t_1}^{t_2} e^{-\lambda a(t)} b(t) dt,
\end{equation*}
with $a$, $b$ functions $[t_1,t_2]\to \Real$ satisfying:
\begin{enumerate}
\item $a$ is smooth in a neighborhood of $t_1$ and there exists $\mu>0$ and $a_0\neq0$ such that for $t\to t_1$:
        \begin{equation*}
             a(t)=a(t_1)+a_0(t-t_1)^\mu + \mathcal{O}(|t-t_1|^{\mu+1}).
        \end{equation*}
\item $b$ is smooth in a neighborhood of $t_1$ and there exists $\nu\geq1$ and $b_0\neq0$ such that for $t\to t_1$:
        \begin{equation*}
             b(t)=b_0(t-t_1)^{\nu-1} + \mathcal{O}(|t-t_1|^{\nu}).
        \end{equation*}
\item $t_1$ is a global minimum for $a$ on $[t_1,t_2]$, i.e. $a(t)>a(t_1)$ $\forall t \in ]t_1,t_2[$, moreover for all $\epsilon>0$,
        \begin{equation*}
            \inf_{t\in[t_1+\epsilon,t_2[}\{a(t)-a(t_1)\}>0
        \end{equation*}
\item The integral $I(\lambda)$ converges absolutely for sufficiently large $\lambda$. \\
\end{enumerate}
Then, as $\lambda\to\infty$, we have:
\begin{equation}
\label{Lapexp}
I(\lambda) = e^{-\lambda a(t_1)} \cdot \frac{\Gamma\left(\frac{\nu}{\mu}\right)}{\lambda^{\nu/\mu}} \cdot \frac{b_0}{\mu\cdot a_0^{\nu/\mu}} \left(1+\mathcal{O}(\lambda^{-(1+\nu)/\mu})\right).
\end{equation}
\label{Laplacesmethod}
\end{thm}


\subsection{Main characters}
\begin{df} \label{def Segre}
For $k\leq n$ positive integers, the \emph{Segre zonoid} is the convex body $C(k,n)$ defined as follow. Take $p_1, \ \ldots, \ p_m$ uniformly and independently at random on $S^{k}\times S^{n-k-1}\subset\Real^{d_{k,n}}$ and construct the Minkowski sum $K_m:= \frac{1}{m} \sum_{i=1}^m \left[ 0,\ p_i \right]$. Then $K_m$ converges (w.r.t the Haussdorff metric) almost surely as $m$ goes to infinity, and $C(k,n)$ is defined to be its limit.
\end{df}

The fact that this sequence of random compact sets converges almost surely follow from a strong law of large number that one can find in \cite{LLNforSets}. In the language of the previous section, $C(k,n)$ is the Vitale zonoid associated to $S^k\times S^{n-k-1}$.

\begin{rk}
There is an appropriate notion of tensor product for zonoids, see~\cite[Section 3]{AubQuant}. In this sense the Segre zonoid is a tensor of balls.
\end{rk}

If we think of $\Real^{d_{k,n}}$ as the space of $(k+1)\times(n-k)$ matrices, it turns out that the convex body $C(k,n)$ depend only on the \emph{singular values} of these matrices. We then have \cite[Theorem5.13]{PSC}
\begin{prop} \label{vol Segre}
The volume of the Segre zonoid is given by
\begin{equation*}
|C(k, n)| =\frac{2^{d_{k,n}}\cdot \pi^{(k+1)(2n+4-k)}}{d_{k,n}\cdot\Gamma\left(\frac{k+1}{2}\right)\Gamma\left(\frac{k}{2}\right)\cdots \Gamma\left(\frac{1}{2}\right)\cdot\Gamma\left(\frac{n-k}{2}\right)\Gamma\left(\frac{n-k-1}{2}\right)\cdots \Gamma\left(\frac{n-2k}{2}\right) }  \ I_k(n)
\end{equation*}
where

\begin{equation} \label{Ikn}
I_k(n):=\int_{S^{k}_+}{\left(p_k \cdot (r)^{(k+1)}\right)^{(n-k)}q_k \ dS^{k}}.
\end{equation}
With the functions of the coordinates $x=(x_1,\ldots,x_{k+1})\in\Real^{k+1}$ 
\begin{align} \label{def pk qk}
p_k(x):=\prod_{i=1}^{k+1}x_i &, & q_k(x):=p_k(x)^{-(k+1)}\prod_{i<j} \left|x_i^2-x_j^2 \right|,
\end{align}
and where $r$ is the radial function of the convex body in $\R^{(k+1)}$ whose support function is given by \cite[Proposition 5.8]{PSC}:
\begin{equation}
h(x)=\frac{1}{\left(2\pi\right)^{(k+2)/2}}\int_{\Real^{k+1}}\sqrt{ x_1^2 \xi_1^2+\cdots +x_{k+1}^2 \xi_{k+1}^2}\ e^{-\frac{||\xi_i||^2}{2}} \diff\xi
\label{defh}
\end{equation}
and the domain of integration is
\begin{equation}
S^{k}_+:=\left\{x\in\Real^{k+1} \ |  \ ||x||=1,\  x_1\geq \cdots \geq x_{k+1}\geq 0 \right\}. \label{Sk+}
\end{equation}
\end{prop}

Let us recall the following~\cite[Lemma 5.10]
{PSC}.
 
\begin{prop}The maximum of  the radial  function  $r$ is
\begin{equation*}
R:=r(\mu)=\max_{u\in S^k}r(u)=\frac{1}{\sqrt{\pi}\sqrt{k+1}}\frac{\Gamma\left(\frac{k+2}{2}\right)}{\Gamma\left(\frac{k+1}{2}\right)}.
\end{equation*}
Moreover $\mu$ is a global maximum on $S^k_+$ and the same is true for the function $p_k$ defined in~\eqref{def pk qk}.
 \label{critofpk} \label{propR}
\begin{proof}
For the first part we refer to~\cite{PSC}. 
Consider $p_k$ as a function on the whole space $\Real^{k+1}$. The $i$\textsuperscript{th} component of the gradient $\nabla p_k$ at the point $x$ is $x_1\ldots \hat{x_i} \ldots x_{k+1}$ (the product of all coordinates except $x_i$). This is normal to the sphere if and only if there is $\lambda\in\Real\backslash\{0\}$ such that 
\begin{equation} \label{gradpki}
\forall i \  x_1\ldots \hat{x_i} \ldots x_{k+1}=\lambda x_i.
\end{equation} 
We see that if one of the $x_i$ is zero then they must all be. Thus if $x\in S^k$ we can assume $x_i\neq 0 \ \forall i$ and multiply both side of~\eqref{gradpki} by $x_i$. We obtain that $x$ is a critical point of $p_k$ restricted to $S^k$ (i.e. $\nabla p_k (x)$ is normal to $S^k$) if and only if $x_i=\pm x_j$ for all  $1\leq i, j \leq k+1$ and $\mu$ is the only point with this property in $S^k_+$. Moreover $\mu$ is a maximum because $\nabla p_k (\mu)$ is pointing outward of the sphere. 
\end{proof}
\end{prop}

We will also need to define the following number.

\begin{df} \label{def Lambdak}
To each $a=(a_1,\ \ldots\ ,\ a_k)\in\Real^{k}$ we associate the polynomial of degree $k+1$ $p_a(X):=X^{k+1}+a_1 X^{k-1}-a_2 X^{k-2}+\cdots \pm a_k$ (note the absence of the term of degree $k$). Let $\mathcal{R}_k:=\left\{ a\in\Real^k \ | \ \text{ all the roots of $p_a$ are reals} \right\}$. Then 
\begin{equation*}
\Lambda_k := \int_{\mathcal{R}_k}e^{a_1}\   da.
\end{equation*}
\end{df}

The number $\Lambda_k$ has an other expression if we see things from the \emph{roots} point of view. For that purpose we introduce the square root of the \emph{discriminant} in $\Real^{k+1}$:
\begin{equation*}
\sqrt{\Delta}:=\prod_{i<j} (x_i-x_j) \ \ \forall x=(x_1, \ \ldots \ , x_{k+1})\in\Real^{k+1}.
\end{equation*}
We also let $\mu:=\frac{1}{\sqrt{k+1}}(1,\ldots,1)\in S^k \subset \Real^{k+1}$ and $F_k:=\left\{ x_1\geq x_2 \geq \cdots \geq x_{k+1} \right\}$. Note that on $F_k$, $\sqrt{\Delta}$ is non negative so the notation makes sense.

\begin{prop} \label{old to new Lambda}
For all $k$ positive integer
\begin{equation*}
\Lambda_k =  \Gamma\left( \frac{K+k}{2}\right) \frac{2^{\frac{K+k-2}{2}}}{\sqrt{k+1}} \int_{F_k\cap S^k\cap \mu^\perp} \sqrt{\Delta} \ dS^{k-1}
\end{equation*}
where $K=\binom{k+1}{2}$ and $dS^{k-1}$ is the standard spherical measure of the sphere of $\mu^\perp$ embedded in $\Real^{k+1}$.
\begin{proof}
First by a spherical change of coordinates and by homogeneity (of degree $K$) of $\sqrt{\Delta}$ we have
\begin{equation}\label{eq:LambdaExp}
\int_{F_k\cap \mu^\perp}e^{-\frac{\|v\|^2}{2}} \sqrt{\Delta} \ dv=  \Gamma\left( \frac{K+k}{2}\right) 2^{\frac{K+k-2}{2}} \int_{F_k\cap S^k\cap \mu^\perp} \sqrt{\Delta} \ dS^{k-1}
\end{equation}
where $dv$ is the flat Lebesgue measure on $\mu^\perp$ induced by its embedding in $\Real^{k+1}$.

On another hand, let us introduce the fundamental symmetric polynomials 
\begin{align*}
\sigma_1		&= x_1+\ \cdots \ + x_{k+1}	\\
\sigma_2		&=\sum_{i<j} x_i x_j			\\
\vdots		&						\\
\sigma_{k+1}	&= x_1 \cdots x_{k+1}.
\end{align*}
This is a change of variables on $F_k$ whose Jacobian is precisely $\det\left(\frac{\partial \sigma}{\partial x}\right)= \sqrt {\Delta}$. In fact,  $\det\left(\frac{\partial \sigma}{\partial x}\right)$ is a monic polynomial of the same degree of $\sqrt{\Delta}$; moreover, it is easy to see that for every $i\neq j$ the polynomial $(x_i-x_j)$ divides $\det\left(\frac{\partial \sigma}{\partial x}\right)$, therefore they are equal.

Now consider a new orthonormal basis in $\Real^{k+1}$ with first unit vector $\mu$. Let $\tilde{x}$ be the coordinates in this new basis and let $v=(\tilde{x}_2,\ldots, \tilde{x}_{k+1})$.
Observe that $\tilde{x}_1= \langle x, \mu\rangle=\sigma_1/\sqrt{k+1}$. Thus we can write the Jacobian matrix
\begin{equation*}
\frac{\partial \tilde{x}}{\partial \sigma}=\left( \begin{matrix} 	\frac{1}{\sqrt{k+1}}				& 0	& \cdots							&0 	\\
																		&	&								&	\\
											\frac{\partial v}{\partial \sigma_1}	&	&\frac{\partial v}{\partial \sigma_{\geq2}}	& 	\\
																		&	&								&	\\
																			\end{matrix} \right).
\end{equation*} 
This implies that $\det\left(\frac{\partial \tilde{x}}{\partial \sigma} \right)= \frac{1}{\sqrt{k+1}}\det\left(\frac{\partial v}{\partial \sigma_{\geq2}} \right)$. On another hand $\tilde{x}$ is an orthogonal transformation of $x$ so $\det\left(\frac{\partial \tilde{x}}{\partial \sigma} \right)=\det\left(\frac{\partial {x}}{\partial \sigma} \right)=1/\sqrt{\Delta}$. Altogether this gives
\beq\label{eq:dvds}
\det\left(\frac{\partial v}{\partial \sigma_{\geq2}} \right)=\sqrt{k+1}/\sqrt{\Delta}
\eeq

Moreover we see that $(\sigma_1)^2=\|x\|^2+2\sigma_2$. Restricted to $v\in\mu^\perp=\{\sigma_1=0\}$ this gives $-\frac{\|v\|^2}{2}=\sigma_2$ . 

Next we let $a_i:=\sigma_{i+1}$ for $1\leq i \leq k$ and apply the change of variable $v\to a$ to the left hand side of~\eqref{eq:LambdaExp}. The Jacobian is (by~\eqref{eq:dvds}) $\sqrt{\Delta} \ \dd v=\sqrt{k+1}\, \dd a$. This gives 
\beq
\int_{F_k\cap \mu^\perp}e^{-\frac{\|v\|^2}{2}} \sqrt{\Delta} \ dv=\sqrt{k+1}\,  \Lambda_k
\eeq
From which we conclude.
\end{proof}
\end{prop}

\section{Asymptotics}
Fix an integer $k>0$. 
Given $L\in\Grass(n-k-1,n)$ we have a corresponding \emph{Schubert variety} in the Grassmannian $\Grass(k,n)$:
\begin{equation} \label{def Schub}
\Omega(L):=\left\{\ell\in\Grass(k,n) \ | \ \ell\cap L \neq \emptyset \right\}.
\end{equation}
It is a singular subvariety of codimension $1$ of the Grassmannian\footnote{As a Young tableau, it correspond to a single square in the upper left corner.} and its volume is computed in~\cite{PSC}.

Recall that we are interested in the computation of the numbers
\begin{equation}
\delta_{k,n}:=\Esp{ \# \left\{g_1\cdot\Omega\left(L\right) \cap \ldots \cap g_{(k+1)(n-k)}\cdot\Omega\left(L\right) \right\} }.
\end{equation}

for which the following formula is established in~\cite{PSC}: 
\begin{equation} \label{dkn fund}
\delta_{k,n}=\frac{d_{k,n}!}{2^{d_{k,n}}} \cdot |\Grass(k,n)| \cdot |C(k,n)|
\end{equation}
where $C(k,n)$ is the convex body defined in Definition~\ref{def Segre} (here $d_{k,n}=(k+1)(n-k)$ is the dimension of the Grassmanian $\Grass(k,n)$).

Using Proposition~\ref{vol Segre} and \cite[equation (2.11)]{PSC} we get 
\begin{equation}\label{eq1}
\delta_{k,n}=\beta_{k,n}\ I_k(n)
\end{equation}
with $I_k(n)$ defined in~\eqref{Ikn} and
\begin{equation}\label{eq2}
\beta_{k,n}:=(2\pi)^{k+1}\left(\frac{\pi}{2}\right)^{d_{k,n}} \frac{\Gamma\left(d_{k,n}\right)}{\Gamma(\frac{n+1}{2})\Gamma(\frac{n}{2})\cdots\Gamma(\frac{n-2k}{2})}.
\end{equation}


\subsection{Asymptotic of $\delta_{k,n}$ as $n\to\infty$ } \label{sec:Asymp}

In this section we are interested in computing the asymptotic (with $k$ still fixed) of $\delta_{k,n}$ as $n$ goes to $\infty$.

In order to compute this asymptotic we will apply Laplace's Method (Theorem~\ref{thm:Laplace}) to Equation~\eqref{Ikn} using the fact that the global maximum of $p_k \cdot (r)^{(k+1)}$ is reached at $\mu$ (Proposition~\ref{propR}). There are two major obstacles that arise.
 First: we don't know explicitly the radial function $r$. Second: one needs to compute the Hessian of $p_k \cdot (r)^{(k+1)}$.

To solve the first problem the key is Proposition~\ref{gradh} that will allow us to express $r$ in terms of the support function, see Equation~\eqref{rbyh} below.

To deal with the second difficulty we will prove that the Hessian of  $p_k \cdot (r)^{(k+1)}$ is a multiple of the identity. To do so we use the fact that the convex body defined by $r$ is invariant by the action of the symmetric group by permutation of coordinates. This implies that the Hessian is a morphism of representations on an irreducible subspace and we can use Schur's Lemma (see Proposition~\ref{Schur} below).

Let us denote $D(k)$ the convex body defined by $r$ and $\partial D(k)$ its boundary. Using Proposition~\ref{gradh}, we have the following commutative diagram:

\begin{center}
\begin{tikzcd}
S^{k} \arrow[r, "\nabla h"] \arrow[rr, bend right, "\psi"]&\partial D(k) \arrow[r, "\pi"] \arrow[rr, bend left, "\|\cdot\|"] &S^{k} \arrow[r, "r"] & \Real_+
\end{tikzcd}
\end{center}

where $\pi(x)=\frac{x}{||x||}$ and $\psi=\pi\circ\nabla h$. Thus assuming for one moment that $\psi$ is a local diffeomorphism around $\mu$, we can write near this point 

\begin{equation}
r(x)^2=\|(\nabla h) \left(\psi^{-1} (x) \right) \|^2.
\label{rbyh}
\end{equation}
Here the gradient of $h$ is the gradient of the function on the whole space $\Real^{k+1}$ and restricted to the sphere afterward only, but for the sake of simplicity we omit the restriction in the notation of the function.

Thus if we can compute the Taylor polynomial of $\nabla h$ at $\mu$, we would at the same time get the Taylor polynomial of $r$ using the following Lemma.

\begin{lemma}
Let $f_1:\Real^p\to\Real^q$, $f_2:\Real^n\to\Real^p$ and $f_3:\Real^m\to\Real^n$ be $\mathcal{C}^2$ functions. The second derivative of the composition at $0\in\Real^m$ is:
\begin{multline*}
\Diff^2_0(f_1\circ f_2 \circ f_3) (x,x)=\Diff_{f_2(f_3(0))}f_1\cdot\left[\Diff_{f_3(0)}f_2\cdot \Diff^2_0 h(x,x)+\Diff^2_{f_3(0)}f_2\left(\Diff_0 f_3\cdot x, \Diff_0 f_3\cdot x\right) \right] \\
+\Diff^2_{f_2(f_3(0))}f_1\left(\Diff_{f_3(0)}f_2\cdot\Diff_0 f_3\cdot x, \Diff_{f_3(0)}f_2\cdot\Diff_0 f_3\cdot x \right) \ \forall x \in\Real^m.
\end{multline*}
\begin{proof}
Let $f$ and $g$ be $\mathcal{C}^2$ functions between real vector spaces that can be composed. We have the Taylor series: 
\begin{equation*}
g(x)=g(0)+\Diff_{0}g\cdot x +\frac{1}{2}\Diff^2_{0}g(x,x)+\mathcal{O}(||x||^3).
\end{equation*} 
Writing the Taylor series at $g(0)$ for $f(g(x))=f(g(0)+\Diff_{0}g\cdot x +\Diff^2_{0}g(x,x)+\mathcal{O}(||x||^3))$ and putting together the terms of second order we get
\begin{equation*}
\Diff_{0}^2(f\circ g)(x,x)=\Diff_{g(0)}f\cdot\Diff^2_{0} g (x,x) +\Diff^2_{g(0)}f(\Diff_{0} g\cdot x, \Diff_{0} g\cdot x).
\end{equation*}
Replacing $f$ by $f_1$ and $g$ by $f_2\circ f_3$, we get the result.
\end{proof}
\end{lemma}

In particular if $f:\Real^m \to \Real$ and $g:\Real^n\to\Real^m$:

\begin{multline}
\Diff^2_a\left((f \circ g)^2\right) (x,x)=2f(g(a))\cdot\left[\Diff_{g(a)}f\cdot \Diff^2_ag(x,x)+\Diff^2_{g(a)}f\left(\Diff_a g\cdot x, \Diff_a g\cdot x\right) \right] \\
+2\left(\Diff_{g(a)}f\cdot\Diff_a g\cdot x \right)^2 \ \forall x \in\Real^m.
\label{devcomp2}
\end{multline}

    For now on we will work in exponential coordinates at $\mu$. That is for all $x\in T_{\mu}S^{k}=\mu^{\perp}$ corresponds the point $\cos ||x|| \ \mu + \sin ||x|| \ \frac{x}{||x||} \ \in S^k$, in particular in these coordinates $\mu=0$. Thus $\psi$ and $\frac{\partial h}{\partial x_i}|_{S^k}$ can be considered as functions on $\Real^k\cong \mu^\perp$ on these coordinates.

We would like to replace $g$ by $\psi^{-1}$ and $f$ by $\frac{\partial h}{\partial x_i}|_{S^k}$ in~\eqref{devcomp2} and sum over $i\in\left\{1,\ldots,k+1\right\}$ to get $\Diff^2_\mu(r^2)$. The problem being that we would need to compute all the entries of the Hessian matrix which are approximately $k^2$. This increasing complexity could make the computation impossible, however as we pointed out before thanks to the symmetries of the function $h$, it turns out that this matrix is a multiple of the identity.

\begin{prop}
\label{Schur}
If $f:\mu^{\perp}\to\Real$ is $C^2$ in a neighborhood of $0$ and invariant by the standard action of $\grpsym_{k+1}$ on $\Real^{k+1}$ then its Hessian at $0$ is a multiple of the identity, i.e. there is $C_f \in \Real$ such that
\begin{equation*}
\Diff^2_0f(x,x)=C_f ||x||^2 \ \forall x \in \mu^\perp.
\end{equation*}
\begin{proof}
First note that $\mu\in\Real^{k+1}$ is fixed by the standard action of $\grpsym_{k+1}$ (that is by permutation of coordinates). Thus the action decomposes in $\Real^{k+1}=\Real \mu \oplus \mu^{\perp}$ and the action on $\mu^{\perp}$ is well defined. Moreover the invariant subspace $\mu^{\perp}$ is irreducible.

Now, let $P\in\grpsym_{k+1}$. Since $f$ is ${C}^2$ we have
\begin{align*}
f(x)		&=f(0)+\Diff_0f\cdot x + \frac{1}{2} x^T H x +\mathcal{O}\left(||x||^3\right) \\
f(P x)	&=f(0)+\Diff_0f\cdot P x + \frac{1}{2} (P x)^T H (P x) +\mathcal{O}\left(||x||^3\right)
\end{align*}
where we wrote the quadratic form $\Diff_0^2f$ in matrix form: $\Diff_0^2f(x,x)=x^T H x$ for a certain symmetric matrix $H$.

By comparing the terms of order 2 we get $P^T H P = H$. Moreover $\grpsym_{k+1}$ acts by orthogonal matrices thus $H$ commutes with this action. Thus $H$  is a morphism of representation from $\mu^\perp$ onto itself. Since $\mu^\perp$ is irreducible it follows from Schur's Lemma (\cite[Section 1.2]{FultonHarris}that $H$ is a (possibly complex) multiple of the identity. $H$ being symmetric all of its eigenvalues are real thus it is a real multiple of the identity.
\end{proof}
\end{prop}

\begin{rk}
\label{rkCr}
By looking at~\eqref{defh} we see that $h$ is $\grpsym_{k+1}$-invariant, thus $r^2$ is invariant as well and the computation of its Hessian is reduced to the computation of the coefficient $C_{r^2}$ of Proposition \ref{Schur}.
\end{rk}

We write $h_i:=\frac{\partial h}{\partial x_i}|_{S^k}$. We recall that in~\cite{PSC} it is established that $\nabla h (\mu) =\mu$. Thus $\psi(\mu)=\mu$. 
\begin{rk}
\label{rkCpsi}
Furthermore $\psi=\pi\circ\nabla h$ is also $\grpsym_{k+1}$-invariant and (in exponential coordinates at $\mu$) is a function from $\mu^{\perp}$ onto itself. Mimicking the proof of Proposition~\ref{Schur}, we get that if the differential of $\psi$ at $\mu$ has at least one real eigenvalue $C_\psi$, then for all $x \in \mu^{\perp}$ we have $\Diff_\mu \psi \cdot x= C_\psi x$. 
\end{rk}

Before stating the next results, we write $h$ in a more convenient form to work with by change of variables (spherical coordinates) in~\eqref{defh}.
\begin{equation} \label{hspherical}
h(x)=\frac{2^k}{\pi^{(k+2)/2}} \Gamma\left(\frac{k+2}{2}\right)\int_{S^k_{pos}}\sqrt{\sum_{j=1}^{k+1}x_j^2\xi_j^2}\ dS^k(\xi)
\end{equation}
with $S^k_{pos}=S^k\cap \left\{x_i\geq0 \ \forall i=1,\ldots,k+1\right\}$ (note that $S^k_{pos}\neq S^k_+$).

\begin{df}For $m\in \mathbb{N}$ we denote by $G(m)$ the numbers:
\begin{equation*}
G(m):=\int_{S^k_{pos}}\xi_1^m dS^k(\xi).
\end{equation*} 
\end{df}
These numbers satisfy the following simple identities.
\begin{prop} \label{usid}
\begin{enumerate}
\item $G(m)= \frac{\pi^{k/2}}{2^k} \frac{\Gamma\left(\frac{m+1}{2}\right)}{\Gamma\left(\frac{k+m+1}{2}\right)}$
\item $\frac{G(4)}{G(2)}=\frac{3}{k+3} $
\item $\frac{G(6)}{G(2)}=\frac{15}{(k+3)(k+5)} $
\item $\frac{\int_{S^k_{pos}}\xi_1^2 \xi_2^2 dS^k(\xi)}{G(2)}=\frac{1}{k+3}$
\item $\frac{\int_{S^k_{pos}}\xi_1^4 \xi_2^2 dS^k(\xi)}{G(2)}=\frac{3}{(k+3)(k+5)}. $
\end{enumerate}
\end{prop}
\begin{proof}Observe first that for any $p>0$
\begin{equation} \label{eq:OtherGamma}
\int_0^{+\infty} t^p e^{-\frac{t^2}{2}} dt =2^{\frac{p-1}{2}} \Gamma\left(\frac{p+1}{2} \right).
\end{equation}
(Use the change of variable $u=\frac{t^2}{2}$ and the definition of the Gamma function).
We prove the first two points, all the other points are done in a similar way.
\begin{enumerate}
\item Using a polar change of variables and Equation~\eqref{eq:OtherGamma} we have
\begin{equation*}
\int_{\Real^{k+1}_{pos}}\xi_1^m e^{-\frac{||\xi||^2}{2}} d\xi = \Gamma\left(\frac{k+m+1}{2}\right)2^{\frac{k+m-1}{2}} G(m).
\end{equation*}
where $\Real^{k+1}_{pos}=\left\{x_i\geq0 \ \forall i=1,\ldots,k+1\right\} $ is the positive orthant. In an other hand using Fubini
\begin{equation*}
\int_{\Real^{k+1}_{pos}}\xi_1^m e^{-\frac{||\xi||^2}{2}} d\xi =  \left(\int_0^{+\infty}e^{-\frac{t^2}{2}}\right)^k  \left(\int_0^{+\infty}t^m e^{-\frac{t^2}{2}}\right)= \left( \frac{\pi}{2}\right)^{\frac{k}{2}} \ 2^{\frac{m-1}{2}} \Gamma\left(\frac{m+1}{2} \right).
\end{equation*}
Equaling the right hand side of this two equalities gives us the result.

\item Using 1. we have $\frac{G(4)}{G(2)}=  \frac{\Gamma\left(\frac{5}{2}\right)}{\Gamma\left(\frac{3}{2}\right)} \frac{\Gamma\left(\frac{k+3}{2}\right)}{\Gamma\left(\frac{k+5}{2}\right)}$.
But $\Gamma\left(\frac{k+5}{2}\right)=\frac{k+3}{2}\Gamma\left(\frac{k+3}{2}\right)$ and $\frac{\Gamma\left(\frac{5}{2}\right)}{\Gamma\left(\frac{3}{2}\right)}=3/2$. That gives us the result.
\end{enumerate} 
 \end{proof}
 
\begin{rk}
Observe that the coefficient in front of the integral in~\eqref{hspherical} is  $\frac{R}{\sqrt{k+1}}\frac{1}{G(2)}$.
\end{rk}

\begin{prop}
\label{propCpsi}
$\Diff_\mu\psi$ admits a real eigenvalue $C_\psi=\frac{k+1}{k+3}$. Thus (by Remark~\ref{rkCpsi}) it is a non zero multiple of the identity. In particular $\psi$ is  a local diffeomorphism near $\mu$.
\begin{proof}
First of all, in~\eqref{hspherical} we integrate a bounded function over a compact domain, thus in computing $h_i$ we can interchange integral and derivative. We get
\begin{equation}
h_i(x)=\frac{R}{\sqrt{k+1}}\frac{1}{G(2)}\int_{S^k_{pos}}\frac{x_i\ \xi_i^2}{\sqrt{\sum_{j=1}^{k+1}x_j^2\xi_j^2}}\ dS^k(\xi).
\end{equation}
Let $\gamma$ be the geodesic on the sphere starting at $\mu$ with initial velocity $\dot{\gamma}_0=\frac{\sqrt{k}}{\sqrt{k+1}}(1,-\frac{1}{k},\ldots,-\frac{1}{k})$, i.e.
\begin{equation}
\begin{split}
\label{defgamma}
\gamma: \Real	&\to S^k \\
		t	&\mapsto\frac{1}{\sqrt{k+1}}\left(\cos t +\sqrt{k}\sin t, \cos t -\frac{\sin t}{\sqrt{k}},\ldots, \cos t -\frac{\sin t}{\sqrt{k}}\right).
\end{split}
\end{equation}
Along this particular geodesic, 
\begin{equation}
h_1(\gamma(t))=\frac{R}{G(2)\sqrt{k+1}}\int_{S^k_{pos}}\frac{(\cos t +\sqrt{k}\sin t)\ \xi_1^2}{\sqrt{\xi_1^2(\cos t +\sqrt{k}\sin t)^2+(1-\xi_1^2)(\cos t -\frac{\sin t}{\sqrt{k}})^2}}\ dS^k(\xi),
\end{equation}
which we can expand as:
\begin{align}
h_1(\gamma(t))&=\frac{R}{G(2)\sqrt{k+1}}\int_{S^k_{pos}}\left[ \xi_1^2+t \ \frac{k+1}{\sqrt{k}}(\xi_1^2-\xi_1^4)\right] dS^k(\xi)+\mathcal{O}(t^2) \\
                &=\frac{R}{G(2)\sqrt{k+1}}\left[ G(2)+t \ \frac{k+1}{\sqrt{k}}\left(G(2)-G(4)\right)\right] +\mathcal{O}(t^2)
\label{higtint}
\end{align}
Using Proposition~\ref{usid}, we find
\begin{equation}
h_1(\gamma(t))=\frac{R}{\sqrt{k+1}}\left[ 1+t \ \sqrt{k} \ \frac{k+1}{k+3}\right] +\mathcal{O}(t^2).
\label{h1gt}
\end{equation}
Similarly we find for $j\geq2$
\begin{equation}
h_j(\gamma(t))=\frac{R}{\sqrt{k+1}}\left[ 1-t \ \frac{1}{\sqrt{k}} \ \frac{k+1}{k+3}\right] +\mathcal{O}(t^2).
\label{h2gt}
\end{equation}
Moreover 
\begin{equation}
\psi_i (\gamma(t))=\frac{h_i\left(\gamma(t)\right)}{\sqrt{\left(h_1(\gamma(t)) \right)^2+k\left(h_2(\gamma(t)) \right)^2}}.
\end{equation}
Taking once again the first order Taylor polynomial in $t$ we find 
\begin{align*}
\psi_1(\gamma(t))&=\frac{1}{\sqrt{k}}+t\ \frac{\sqrt{k}}{\sqrt{k+1}} \cdot \frac{k+1}{k+3}+\mathcal{O}(t^2) \\
\psi_2(\gamma(t))&=\frac{1}{\sqrt{k}}-t\ \frac{1}{\sqrt{k}\sqrt{k+1}} \cdot \frac{k+1}{k+3}+\mathcal{O}(t^2).
\end{align*}
Recalling that $\Diff_\mu \psi \cdot \dot{\gamma}_0=\frac{\diff}{\diff t}|_{t=0}\psi(\gamma(t))$ we find that $\dot{\gamma}_0$ is an eigenvector with eigenvalue $\frac{k+1}{k+3}$.
\end{proof}
\end{prop}

\begin{rk}
By looking at~\eqref{h1gt}  and~\eqref{h2gt} we note that 
$\Diff_\mu (\nabla h)\cdot\dot{\gamma}_0=\sum_{i=1}^{k+1}\Diff_\mu h_i\cdot\dot{\gamma}_0=0.$
Thus $\dot{\gamma}_0$ is an eigenvector for $\Diff_\mu (\nabla h)$ of eigenvalue  $0$ and by the same argument as in remark~\ref{rkCpsi}, $\Diff_\mu (\nabla h)=0$, i.e. $\mu$ is a critical point of $\nabla h$.
\label{rkgradh}
\end{rk}
\begin{prop} We have: 
\label{propD2r2}
\begin{equation*}
\Diff^2_\mu (r^2) (x,x)= \frac{2}{C_\psi^2}\left(\sum_{i=1}^{k+1}\left(\Diff_\mu h_i\cdot x\right)^2 + \frac{R}{\sqrt{k+1}}\sum_{i=1}^{k+1}\Diff^2_\mu h_i (x,x)   \right) \ \forall x \in \mu^{\perp}.
\end{equation*}
\begin{proof}
In Equation~\eqref{devcomp2} take $f=h_i$ and $g=\psi^{-1}$. Sum over $i\in\{1,\ldots,k+1\}$ and use Remark \ref{rkgradh}.
\end{proof}
\end{prop}

We are now finally ready to compute the Hessian of $r^2$.

\begin{prop}\label{prop:Hessr}
For all $x\in\mu^{\perp}$ we have
\begin{equation*}
\Diff^2_\mu (r^2) (x,x)= -4 \frac{R^2}{k+1} \ ||x||^2.
\end{equation*}
\begin{proof}
Once again we use $\gamma$ the geodesic defined by~\eqref{defgamma}, and we compute $\Diff^2_\mu (r^2) (\dot{\gamma}_0,\dot{\gamma}_0)$ using Proposition~\ref{propD2r2}.
With the help of~\eqref{h1gt}  and~\eqref{h2gt}, we get:
\begin{equation}
\sum_{i=1}^{k+1}\left(\Diff_\mu h_i\cdot \dot{\gamma}_0\right)^2=\left(\Diff_\mu h_1\cdot \dot{\gamma}_0\right)^2+k\left(\Diff_\mu h_2\cdot \dot{\gamma}_0\right)^2=R^2 \frac{(k+1)^2}{(k+3)^2}.
\end{equation}
To compute the second derivative of $h_i$ we need to take the Taylor series of equation~\eqref{higtint} up to order 2 this time. The term of order 2 for $h_1$ is $\frac{t^2}{2}\frac{k+1}{k}\left[2\xi_1^2-\xi_1^4(3k+5)+\xi_1^6 3(k+1)\right]$  which once integrated and using the useful relations, gives $\Diff^2_\mu h_1(\dot{\gamma}_0,\dot{\gamma}_0)=-\frac{R}{\sqrt{k+1}}\frac{(k+1)(7k-1)}{(k+3)(k+5)}$.

Similarly, one gets $\Diff^2_\mu h_2(\dot{\gamma}_0,\dot{\gamma}_0)=\frac{R}{\sqrt{k+1}}\frac{(k+1)}{k}\left(-1+9\frac{(k+1)}{(k+3)(k+5)}\right)$. We combine those to obtain
\begin{equation}
\sum_{i=1}^{k+1}\Diff^2_\mu h_i (\dot{\gamma}_0,\dot{\gamma}_0)=\Diff^2_\mu h_1 (\dot{\gamma}_0,\dot{\gamma}_0)+k\Diff^2_\mu h_2 (\dot{\gamma}_0,\dot{\gamma}_0)=-\frac{R}{\sqrt{k+1}} \frac{(k+1)^2}{(k+3)}.
\end{equation}
The result  follows from Proposition~\ref{propD2r2}, Remark~\ref{rkCr} and the fact that $||\dot{\gamma}_0||=1$. 
\end{proof}
\end{prop}

\begin{rk}
The Hessian of the various intermediate functions such as the $h_i$'s depend on the choice of local coordinates and make sense only if we consider them as functions on $\mu^{\perp}$. However $\mu$ being a critical point of $r^2$ its Hessian at this point is well defined and once computed does not depend on the choice of local coordinates.
\end{rk}

All the work is now almost done. We write~\eqref{Ikn}  in Riemannian polar coordinates:
\begin{equation}
I_k(n)=\int_{\tilde{S}^{k-1}_+}\int_0^{l(v)} e^{-(n-k)\left(-\frac{k+1}{2}\log(r^2)-\log(p)\right)}q \sqrt{\det g} \ \rho^{k-1}\ d\rho \ dS^{k-1}(v)
\end{equation}
where $g$ is the spherical metric of $S^k$ on $\mu^\perp$ the angular domain $\tilde{S}^{k-1}_+:=\pi\circ\exp_\mu^{-1}(S^k_+)=\mu^\perp\cap S^k \cap F_k$ and $l(v)$ is the time to reach the boundary of the domain $S^k_+$ starting at $\mu$ with velocity $v$.

In order to apply Theorem~\ref{Laplacesmethod} we need to take the Taylor series of the various functions appearing in the integrand. A simple (but rather tedious) computation leads to:

\begin{multline}
I_k(n)=\frac{2^{K}}{(k+1)^{\frac{K-(k+1)^2}{2}}}\int_{\tilde{S}^{k-1}_+}\prod_{i<j}|x_i-x_j| \\
	\int_0^{l(v)} e^{-(n-k)\left(-\frac{k+1}{2} \log\left(\frac{R^2}{k+1}\right)+\rho^2(k+2)+\mathcal{O}(\rho^3)\right)}(\rho^{K+k-1}+\mathcal{O}(\rho^{K+k}))\ d\rho \ dS^{k-1}
\end{multline}

where $K=\binom{k+1}{2}=\frac{k(k+1)}{2}$.

To apply Laplace's method, let us first prove the following fact.
\begin{lemma} \label{upbound}
For all $v\in\tilde{S}^{k-1}_+$, $[0,\tan^{-1}(1/\sqrt{k})]\subset[0,l(v)]$.
\begin{proof}
$S^k_+$ is the (geodesically) convex hull on $S^k\subset\Real^{k+1}$ of the points $\alpha_1:=(1,0,\ldots,0)$, $\alpha_2:=\frac{1}{\sqrt{2}}(1,1,0,\ldots,0)$, $\ldots$ , $\alpha_{k+1}=\mu$. The closest of these points to $\mu$ (except $\mu$ itself of course) is $\alpha_k$. The cosine of the angle between them is given by their scalar product $\langle \alpha_k, \mu \rangle=\frac{\sqrt{k}}{\sqrt {k+1}} $. The result follows from the formula $\tan(\cos^{-1}(x))=\frac{\sqrt{1-x^2}}{x}$.
\end{proof}
\end{lemma}
 Thus the upper bound $l(v)$ doesn't really matter for the asymptotic. Moreover the outermost integral is the integral of a bounded function on a compact domain and we can interchange it with the limit. We apply Theorem~\ref{Laplacesmethod} with $\lambda=(n-k)$, $\mu=2$ and $\nu=K+k$. We find, using Proposition~\ref{old to new Lambda}:
\begin{equation}
I_k(n)=\frac{\Gamma\left(\frac{K+k}{2}\right)}{\Gamma\left(\frac{K+1}{2}\right)}\frac{2^{\frac{K-1}{2}}\Lambda_k}{(k+1)^{\frac{K-(k+1)^2}{2}}(k+2)^{\frac{K+k}{2}}} \left(\frac{R}{\sqrt{k+1}}\right)^{(n-k)(k+1)}\frac{1}{n^{\frac{K+k}{2}}}\left(1+\mathcal{O}((n-k)^{-\frac{K+k+1}{2}})\right)
\label{Iknasymp}
\end{equation}

We are now (finally) ready to state the main theorem of this section.

\begin{thm}
For every fixed integer $k>0$ and as $n$ goes to infinity, we have
\begin{equation}
\delta_{k,n}=a_k \cdot \left(b_k\right)^n\cdot n^{-\frac{k(k+1)}{4}}\left(1+\mathcal{O}(n^{-1})\right)
\label{dknasymp}
\end{equation}
where
\begin{align}
a_k &=\Lambda_k \  \frac{2^{\frac{k(k-3)}{4}}}{\pi^{\frac{k(k+2)}{2}}}\sqrt{k+1}\ \left( \frac{k+1}{k+2}\right)^{\frac{k(k+3)}{4}}\left(\frac{\Gamma\left(\frac{k+1}{2} \right)}{\Gamma\left(\frac{k+2}{2} \right)}\right)^{k(k+1)}         \label{defak} \\
b_k&=\left( \frac{\Gamma\left(\frac{k+2}{2} \right)}{\Gamma\left(\frac{k+1}{2} \right)}\sqrt{\pi}\right)^{(k+1)}.
\end{align}
\begin{proof}
We use \eqref{eq1} and \eqref{eq2}. We need to compute the asymptotic of 
\begin{equation}
\beta_{k,n}:=\frac{\pi^{(k+1)(n-k)+(k+1)}}{2^{(k+1)(n-k-1)}} \frac{\Gamma\left((k+1)(n-k)\right)}{\Gamma(\frac{n+1}{2})\Gamma(\frac{n}{2})\cdots\Gamma(\frac{n-2k}{2})}.
\end{equation}
We use the multiplication theorem to write:
\begin{equation*}
\Gamma\left((k+1)(n-k)\right)=\frac{(k+1)^{(k+1)(n-k)-1/2}}{(2\pi)^{k/2}}\prod_{l=0}^k \Gamma\left(n-k+\frac{l}{k+1} \right)
\end{equation*}
and the denominator
\begin{equation*}
\prod_{l=0}^k\Gamma\left(\frac{n+1-2l}{2} \right)\Gamma\left(\frac{n-2l}{2} \right)=\frac{\pi^\frac{k+1}{2}}{2^{(k+1)(n-k-1)}}\prod_{l=0}^k\Gamma\left(n-2l \right).
\end{equation*}
Moreover using~\eqref{Frac Gam asymptote} we have:
\begin{equation*}
\prod_{l=0}^k \frac{\Gamma\left(n-k+\frac{l}{k+1} \right)}{\Gamma\left(n-2l \right)}=\prod_{l=0}^k n^{-k+\frac{l}{k+1}+2l}\ \left(1+\mathcal{O}(n^{-1}) \right)= n^{k/2} \ \left(1+\mathcal{O}(n^{-1}) \right).
\end{equation*}
Thus
\begin{equation}
\beta_{k,n}= \frac{ \left(\pi (k+1) \right)^{ (k+1)(n-k) }}{(2\pi)^{k/2}  \sqrt{k+1}}\  \  n^{k/2}\  \left(1+\mathcal{O}(n^{-1})\right).
\end{equation}
Reintroducing it carefully into $\delta_{k,n}=\beta_{k,n}\cdot I_k(n)$ and using~\eqref{Iknasymp} and Proposition~\ref{propR} we get the result.
\end{proof}
\end{thm}

We notice the structure of formula~\eqref{dknasymp}: it consists of a factor $a_k$ that does not depend of $n$, another factor $(b_k)^n$ that grows exponentially fast and a \emph{rational} one $n^{-k(k+1)/4}$. The last two are easily computable for any $k>0$. Unfortunately the expression of $a_k$ in~\eqref{defak} still depends on the constant $\Lambda_k$ for which there is little hope to find a closed formula for all $k>0$. However some particular values can be computed explicitly.

\begin{prop}
The first two values of $\Lambda_k$ are
\begin{align*}
\Lambda_1 &= 1 \\
\Lambda_2 &= \sqrt{\frac{\pi}{3}} .
\end{align*}
\end{prop}

\begin{proof}
We use directly the Definition~\ref{def Lambdak}.

For $k=1$ the polynomial $X^2+a$ has real roots if and only if $a\leq 0$. Thus 
\begin{equation*}
\Lambda_1=\int_{0}^{+\infty}e^{-t} \ dt = 1.
\end{equation*}

For $k=2$ the polynomial $X^3+aX-b$ has all its roots in $\Real$ if and only if the discriminant $\Delta= -4 a^3-27 b^2$ is positive. For fixed $a=-t$ this means $b^2\leq \frac{4}{27}\ t^3$ i.e. $b\in\left[ -\frac{2}{3\sqrt{3}} t^{3/2}, +\frac{2}{3\sqrt{3}} t^{3/2}\right]$. Thus 
\begin{equation*}
\Lambda_2=\frac{4}{3\sqrt{3}}\int_{0}^{+\infty}t^{3/2} e^{-t} \ dt = \frac{4}{3\sqrt{3}} \Gamma\left(\frac{5}{2} \right)= \frac{4}{3\sqrt{3}}\  \frac{3\sqrt{\pi}}{4}.
\end{equation*}
\end{proof}

This allows us to write down explicitly the first three asymptotic values of $\delta_{k,n}$
\begin{align}
\delta_{0,n}&=1 \\
\delta_{1,n}&= \frac{8}{3\pi^{5/2}} \cdot \left(\frac{\pi^2}{4}\right)^n \cdot n^{-1/2} \left( 1+\mathcal{O}\left(n^{-1}\right)\right)\\
\delta_{2,n}&= \frac{9\sqrt{3}}{2048\sqrt{2\pi}} \cdot 8^n \cdot n^{-3/2} \left( 1+\mathcal{O}\left(n^{-1}\right)\right).
\end{align}

\subsection{The asymptotic in the complex case}\label{sec:Cpx}

One can state the same problem over the \emph{complex} Grassmannian of complex subspace of $\CP^n$.

Recall from the introduction (equation~\eqref{dkncIntro}) that we denote by $\delta_{k,n}^\C$ the number of complex $k$-subspaces of $\CP^n$ meeting $(k+1)(n-k)$ generic subspaces of dimension $n-k-1$.
A closed formula is known for $\delta_{k,n}^\C$ for every $k,n$ (see~\cite[Corollary 4.15]{PSC}):
\beq \label{dknCexact}
\delta_{k,n}^\C = \frac{\Gamma(1)\Gamma(2)\cdots \Gamma(k+1)}{\Gamma(n-k+1)\Gamma(n-k+2)\cdots \Gamma(n+1)}(k+1)(n-k) \Gamma\left( (k+1)(n-k) \right)
\eeq

 We can compute its asymptotic.

\begin{prop}For every $k$ fixed, as $n\to \infty$  we have
\beq
\delta_{k,n}^\C = a_k^\C \cdot \left(b_k^\C\right)^n\cdot n^{-\frac{k(k+2)}{2}}\left(1+\mathcal{O}(n^{-1})\right)
\label{dknCasymp}
\eeq
where
\begin{align}
a_k^\C 	&=\frac{\Gamma(1)\Gamma(2)\cdots \Gamma(k+1)}{(2\pi)^{k/2}(k+1)^{k(k+1)-1/2}}         \label{defakC} \\
b_k^\C	&=\left(k+1\right)^{(k+1)}.
\end{align}
\begin{proof}
Using the Multiplication Theorem (Proposition~\ref{prop:MultThm}) we get
\beq
\Gamma\left( (k+1)(n-k) \right)=\frac{(k+1)^{(k+1)(n-k)-1/2}}{(2\pi)^{k/2}}\prod_{i=0}^k \Gamma\left(n-k+\frac{i}{k+1} \right)
\eeq
When reintroduced in Equation~\eqref{dknCexact} this gives
\beq\label{thiseq}
\delta_{k,n}^\C =a_k^\C (b_k^\C)^n (n-k) \prod_{i=0}^k\frac{\Gamma\left(n-k+ \frac{i}{k+1} \right)}{\Gamma\left(n-k+i+1 \right)}
\eeq

Using Proposition~\ref{prop:GammaAsymp} we can deduce that

\begin{align}
 \prod_{i=0}^k\frac{\Gamma\left(n-k+ \frac{i}{k+1} \right)}{\Gamma\left(n-k+i+1 \right)}		&= \prod_{i=0}^k (n-k)^{-i k/(k+1)-1} \left(1+\mathcal{O}(n^{-1})\right) \\
																	&= n^{-k(k+2)/2-1} \left(1+\mathcal{O}(n^{-1})\right)
\end{align}

which once reintroduced in~\eqref{thiseq} gives the result.
\end{proof}
\end{prop}

\section{Periods}\label{sec:periods}
\begin{prop}
For all integers $n>0$ we have $\delta_{0,n}=1$.

\begin{proof}
We look back at the definition of the Schubert variety in~\eqref{def Schub}. In the case $k=0$ we fix $L\in\Grass(n-1,n)$ i.e. an hyperplane of $\Real\Prj^n$. Then $\Omega(L)=\left\{p\in\Real\Prj^n \ | \ p \in L\right\}=L$. Thus $\delta_{0,n}$ is the average number of points in the intersection of $n$ random hyperplanes of $\Real\Prj^n$ that is precisely $1$.
\end{proof}
\end{prop}

In general $\delta_{k,n}$ is a \emph{period} in the sense of Kontsevich-Zagier. In order to prove it we need first the following Lemma.
\begin{lemma}Let $S\subset \R^d$ be a compact semialgebraic set with defining polynomials over $\mathbb{Q}$ and $\alpha:S\to \mathbb{R}$ be an algebraic function with coefficients in $\mathbb{Q}$. Denoting by $\mathrm{vol}_S$ the volume density on the set $\mathrm{sm}(S)$ of smooth points of $S$ associated to the Riemannian metric induced by the ambient space $\mathbb{R}^d$, then
\beq \textrm{$\int_{S}\alpha\cdot \mathrm{vol}_S$ belongs to the period ring}.\eeq
\end{lemma}
\begin{proof}Recall first that the above integral equals, by definition:
\beq \int_{S}\alpha\cdot \mathrm{vol}_S=\int_{\textrm{sm}(S)}\alpha \cdot \mathrm{vol}.\eeq
We now preliminary decompose $S$ into smaller pieces, each with defining polynomials with coefficients in $\mathbb{Q}$ over which we will perform the integral. To this end, observe that:
\beq S=\bigcup_{i=1}^a\bigcap_{j=1}^b\left(\{f_{i,j}=0\}\cap \{g_{i,j}<0\}\right),\eeq
with $f_{i,j}, g_{i,j}\in \mathbb{Q}[x_1, \ldots, x_d].$ Removing all the inequalities from the previous description, assuming each $f_{i,j}$ is irreducible, keeping only the $f_{i,j}$ whose zero set is $s$-dimensional and relabeling these with $f_{i,j}=f_k$, $k\in\{1, \ldots, \gamma\}$, we can set $Y_{k}=\{f_k=f_{i,j}=0\}$ and we see that there exists a semialgebraic set $\Sigma_1$ of dimension $\dim(\Sigma_1)<s$ such that:
\beq S\backslash\Sigma_1\subseteq\bigcup_{k=1}^\gamma Y_{k}.\eeq
By construction now each $Y_k$ has dimension $s$; for every $k=1, \ldots, \gamma$ denote by $X_k$ the set of singular points of $Y_k$ and consider $\Sigma_2=\Sigma_1\cup\left(\bigcup_{k=1}^\gamma X_k\right)$. Then $S\backslash \Sigma_2$ (which coincides with $S$ up to a set of dimension strictly less than $s$) is contained in
\beq S\backslash\Sigma_2\subset\bigcup _{k=1}^\gamma \mathrm{sm}(Y_k).\eeq
For every $k=1, \ldots, \gamma$, because $Y_k$ is smooth, there exist $1\leq i_1<\cdots<i_s\leq d$ such that the critical points of the projection $\mathrm{proj}_{\mathrm{span}\{e_{i_1}, \ldots, e_{i_s}\}}$ restricted to $Y_k$ are a set $C_k$ of codimension one in $Y_k$. Denote by
\beq \Sigma_3=\Sigma_1\cup\Sigma_2\cup\left(\bigcup_{k=1}^\gamma C_k\right)\eeq
(a set of dimension at most $s-1$). Denote also by $L_k=\mathrm{span}\{e_{i_1}, \ldots, e_{i_s}\}$ and by $p_k=\mathrm{proj}_{\mathrm{span}\{e_{i_1}, \ldots, e_{i_s}\}}|_{Y_k}$. We now decompose $S\backslash \Sigma_3$ into \emph{disjoint} pieces $S_k=S\cap \Sigma_3^c$:
\beq S\backslash \Sigma_3=\coprod_{k=1}^\gamma S_k.\eeq
Because $\dim(\Sigma_3)<s$, we have:
\beq \int_{S}\alpha\cdot \mathrm{vol}_S=\sum_{k=1}^\gamma\int_{S_k}\alpha \cdot\mathrm{vol}_S.\eeq

Summing up: the desired integral can be written as a sum of integrals over semilagebraic sets $S_1, \ldots, S_\gamma$, each of dimension $s$, each defined by polynomial equalities and inequalities with coefficients over $\mathbb{Q}$ and with the property that there exists a map $p_k:S_k\to L_k\simeq \R^s$ (which is defined over $\mathbb{Q}$) which is a diffeomorphism onto its image.

For each $k=1, \ldots, \gamma$ consider the inverse of the projection $\tau_k:p_k(S_k)\to S_k$, which is also a diffeomorphism; it is also semialgebraic and defined over $\mathbb{Q}$.
In particular:
\begin{align} \int_{S}\alpha\cdot \mathrm{vol}_S&=\sum_{k=1}^\gamma\int_{S_k} \alpha\cdot \mathrm{vol}_S\\
&=\sum_{k=1}^\gamma \int_{p_L(S_k)} \alpha(\tau_k(y))\cdot \sqrt{\det\left(J\tau_k(y)J\tau_k(x)^T\right)}\, \mathrm{d}y.
\end{align}
In the previous line: each summand is the integral of an algebraic function defined over $\mathbb{Q}$ and the domain of integration is a full dimensional semialgebraic set in $L\simeq \R^s$ defined over $\mathbb{Q}$. In particular each summand is a period, and therefore the whole integral is a period. This concludes the proof.
\end{proof}
\begin{lemma}Let $Z\subset \R^d$ be a compact semialgebraic set with defining polynomials over $\mathbb{Q}$ and let $C_Z$ be the Vitale zonoid associated to $Z$. Then $\mathrm{vol}(C_Z)$ belongs to the period ring.
\end{lemma} 
\begin{proof}
We apply the previous Lemma with the choice of $S=Z\times \cdots \times Z\subset \mathbb{R}^{d\times d}$ and
\beq \alpha(z_1, \ldots, z_d)=\sqrt{\det\left([z_1, \ldots, z_d][z_1, \ldots, z_d]^T\right)}.\eeq

\end{proof}
\begin{cor}\label{cor:periods1}Each $\delta_{k,n} $ belongs to the period ring.
\end{cor}
\begin{proof}
We use Equation~\eqref{dkn fund} and~\eqref{eq:volG}. Since periods form a ring and values of the Gamma function at rational points are periods this proves the statement. 
\end{proof}

\section{A Line Integral for $\delta_{1,n}$}\label{sec:lineInt}
In the case of $\delta_{1,n}$ we can prove the following formula.
\begin{prop}\label{prop:periods2}
\beq
\delta_{1,n}=-2 \pi^{2n-2}c(n)\int_{0}^1{ L(u)^{n-1}\frac{\dd}{\dd u}\left(\cosh\left(w(u)\right)\right)\,\dd u}
\eeq
where
\beq
c(n)=\frac{\Gamma\left(2n-2\right)}{\Gamma\left(n\right)\Gamma\left(n-2\right)}=\frac{n(n-2)}{2}\delta_{1,n}^\C
\eeq
 $L=F\cdot G$ and $w=\log(F/G)$ with
\begin{align}
F(u)&:=\int_{0}^{\pi/2} \frac{u \ \sin^2(\theta)}{\sqrt{\cos^2\theta+u^2 \sin^2\theta}}\mathrm{d}\theta	\\
G(u)&:=\int_{0}^{\pi/2} \frac{ \sin^2(\theta)}{\sqrt{\sin^2\theta+u^2 \cos^2\theta}}\mathrm{d}\theta
\end{align}
\begin{proof}
From~\cite[Equation (6.12)]{PSC} we have 
\beq
\delta_{1,n}= \pi^{2n-2}c(n)\int_0^{\pi/4}\left(r(\theta)^2\, \cos\theta \sin\theta\right)^{n-1} \frac{(\cos\theta)^2-(\sin\theta)^2}{(\cos\theta\, \sin\theta)^2}\dd\theta
\label{eq:d1n}
\eeq
Where $c(n)=\frac{\Gamma\left(2n-2\right)}{\Gamma\left(n\right)\Gamma\left(n-2\right)}$.

Moreover we have from~\eqref{rbyh}
\beq
 r(\theta)^2 = | \nabla h (\cos t, \sin t)|^2 
\eeq
where 
\begin{equation}
\tan \theta = \frac{h_y(\cos t, \sin t)}{h_x (\cos t, \sin t)}
\label{eq:tan}
\end{equation}
and $h$ is given by~\eqref{defh} and can be reduced to:
\begin{equation}\label{eq:h1red}
 h(x,y) =\frac{1}{\pi} \int_0^{\pi/2}\sqrt{x^2 \cos^2(\theta)+y^2\sin^2(\theta)}\, \dd \theta
\end{equation}

Let 
$ p(t) = h_y(\cos t, \sin t) $
and
$ q(t) = h_x(\cos t, \sin t)$.
Then
\begin{align}
 \cos^2 \theta &= \frac{q(t)^2}{q(t)^2 + p(t)^2} \\
 \sin^2 \theta &= \frac{p(t)^2}{q(t)^2 + p(t)^2} \\
r(\theta)^2 &= p(t)^2+q(t)^2
\end{align}
So, if we change the variable of integration in \eqref{eq:d1n} to $t$,
then the integrand becomes
\beq
\left((q^2 + p^2)\frac{pq}{p^2+q^2}\right)^{n-1}\frac{q^2-p^2}{p^2+q^2}\frac{(q^2+p^2)^2}{q^2p^2}
\frac{d}{dt} \left( \frac{p(t)}{q(t)} \right)  \frac{q(t)^2}{q(t)^2 + p(t)^2} dt
\eeq
where we have used \eqref{eq:tan} to determine
$d \theta = \frac{d}{dt} \left( \frac{p(t)}{q(t)} \right) \frac{q(t)^2}{q(t)^2 + p(t)^2} \dd t$.
So \eqref{eq:d1n} becomes
\begin{equation}\label{eq:edeg2}
 \delta_{1,n}= \pi^{2n-2} c(n) \int_0^{\pi/2} \left(p(t)q(t)\right)^{n-3}(q(t)^2 - p(t)^2 ) q(t)^2 \frac{d}{dt} \left( \frac{p(t)}{q(t)} \right) dt
\end{equation}
Next we make the change of variables $u= \tan t$. It is not difficult to see that $p(t)=F(u(t))$ and $q(t)=G(u(t))$ using~\eqref{eq:h1red} and the definition of $F$ and $G$ in the proposition. 
The integral becomes:
\begin{equation}\label{eq:edeg2}
\delta_{1,n}= \pi^{2n-2}c(n)\int_0^{1} \left(F(u)G(u)\right)^{n-3}(G(u)^2 - F(u)^2) G(u)^2 \frac{d}{du} \left( \frac{F(u)}{G(u)} \right) du,
\end{equation}
Now we let $L=F\cdot G$ and $H=F/G$. The integrand becomes $L^{n-1} (1/H-H) H'/H$. The last factor suggest to let $w:=\log(H)$. We obtain
\beq
\delta_{1,n}=-2 \pi^{2n-2}c(n)\int_0^{1} L(u)^{n-1} \sinh(w(u)) \frac\dd{\dd u} w(u) \dd u
\eeq
\end{proof}
\end{prop}


\medskip
\printbibliography


\end{document}